\documentclass[12pt,english,reqno]{amsart}
\usepackage[T1]{fontenc}
\usepackage[latin9]{inputenc}
\usepackage{verbatim}
\usepackage{amsthm}
\usepackage{amstext}
\usepackage{amssymb}
\usepackage{cite}
\makeatletter
\numberwithin{equation}{section}
\numberwithin{figure}{section}
\theoremstyle{plain}
\newtheorem{thm}{\protect\theoremname}[section]
  \theoremstyle{definition}
  \newtheorem{defn}[thm]{\protect\definitionname}
  \theoremstyle{definition}
  \newtheorem{example}[thm]{\protect\examplename}
  \theoremstyle{plain}
  
  \theoremstyle{plain}
  
  \theoremstyle{remark}
  
  \theoremstyle{plain}

\usepackage{amscd}

\newtheorem{theorem}{Theorem}[section]
\theoremstyle{plain}

\newtheorem{corollary}[theorem]{Corollary}

\newtheorem{definition}[theorem]{Definition}

\newtheorem{lemma}[theorem]{Lemma}

\newtheorem{proposition}[theorem]{Proposition}
\newtheorem{remark}[theorem]{Remark}

\newcommand{\eqdef}{\stackrel{\mathrm{def}}{=}}

\newcommand{\N}{\mathbb{N}}
\newcommand{\R}{\mathbb{R}}
\newcommand{\C}{\mathbb{C}}

\newcommand{\supp}{\mbox{supp }}

\makeatother

\usepackage{babel}
  \providecommand{\corollaryname}{Corollary}
  \providecommand{\definitionname}{Definition}
  \providecommand{\examplename}{Example}
  \providecommand{\lemmaname}{Lemma}
  \providecommand{\propositionname}{Proposition}
  \providecommand{\remarkname}{Remark}
\providecommand{\theoremname}{Theorem}

\begin{document}

\title{On compact subsets of  Sobolev spaces on manifolds 
}

\author{Leszek Skrzypczak}

\thanks{One of the authors (L.S.) was supported by National Science Center,
Poland, Grant no. 2013/10/A/ST1/00091.}

\address{Faculty of Mathematics \& Computer Science, Adam Mickiewicz University,
ul. Uniwersytetu Pozna\'{n}skiego 4, 61-614 Pozna\'{n}, Poland}

\email{lskrzyp@amu.edu.pl}

\author{Cyril Tintarev}

\thanks{The second author (C.T.) expresses his gratitude to Leszek Skrzypczak
and the Faculty of Mathematics \& Computer Science of Adam Mickiewicz
University, as well as to Simeon Reich, Yehuda Pinchover and the Faculty of Mathematics at Technion, for their kind hospitality. The latter stay was as a Lady Davis Visiting Professor.}

\address{Technion -- Israel Institute of Technology, Haifa 32000, Israel}

\email{tammouz@gmail.com}

\maketitle

\section{Introduction}
It is common that a Sobolev space defined on $\R^m$ has a non-compact embedding into an $L^p$-space, but it has subspaces for which this embedding becomes compact. There are three well known cases of such subspaces, the  Rellich compactness, for a subspace of functions on a bounded domain (or an unbounded domain, sufficiently thin at infinity),  the Strauss compactness, for a subspace of radially symmetric functions in $\R^m$, cf. \cite{Strauss}, and the weighted Sobolev spaces. Known generalizations of Strauss compactness include subspaces of functions with block-radial symmetry \cite{Lions}, subspaces of functions with certain symmetries on Riemannian manifolds, as well as similar subspaces of more general Besov and Triebel-Lizorkin spaces (see \cite{Skr, Skr2, Skrti1}) . In \cite{HebeyVaugon} presence of symmetries is interpreted in terms of the rising critical Sobolev exponent corresponding to the smaller effective dimension of the quotient space. 
In \cite{Skrti1} a necessary and sufficient condition on the group $G$ of isometries of a Riemannian manifold  is provided for compactness of Sobolev embeddings of a subspace of $G$-symmetric functions, but only for the case when the manifold is a homogeneous space. The objective of this paper is to extend this result to general manifolds that admit Sobolev embeddings,  as well as to  study compactness that results from conditions of quasi-symmetric type rather than from symmetries. In particular we study compactness of embedding of subspaces defined by restriction of the number of independent variables, i.e. subspaces of functions of the form $f\circ \varphi$ with a fixed $\varphi$. 

The method of the proof is based on the  property of cocompactness type for non-compact Sobolev embeddings, Lemma~\ref{lem:spotlight} (the "spotlight lemma"). We then verify that suitable symmetry conditions imply conditions of Lemma~\ref{lem:spotlight}, by the following heuristic argument: if the embedding is not compact on a particular sequence, then by the spotlight lemma there is a sequence of balls on the manifold where the sequence does not locally vanish in $L^1$, but thanks to the symmetry condition on the functions, this non-vanishing may extend to too many balls, providing a contradiction.    

In Section~2 we formulate the spotlight lemma for a general class of manifolds that admit Sobolev embeddings, and define orbital discretizations for Riemannian manifolds as well as functions quasisymmetric relative to an orbital discretization. In Section~3 we prove compactness for subspaces of functions that are quasisymmetric with respect to an abstract orbital discretization, in Theorem~\ref{t:main}. From this theorem we derive in Section~4, Theorem~\ref{thm:groupsymm}, a compactness condition for subspaces defined by a group symmetry, and show that it is also necessary. In Section~5 we study subspaces defined by reduction of variables, and give two sufficient conditions for compactness of such subspaces,  Theorem~\ref{thm:levels1} and Theorem~\ref{thm:levels2}. The compactness condition in the latter, formulated for a class of functions with more regular level sets than the former, is similar to that of  Theorem~\ref{thm:groupsymm} and is also necessary. In Section~6 we study compactness of subsets of Sobolev spaces extended by order. Its main results are  Theorem~\ref{thm:extord} and Corollary~\ref{cor:extorder}. In Section~7 we give existence results to two sample variational problems as an illustration of consequences of compact embeddings for subspaces (obviously, this compactness can be employed in a wide range of minimax problems for quasilinear elliptic PDE). 


\section{Preliminaries: discretization of a manifold and a ``spotlight'' lemma}
Let $M$ be
an $m$-dimensional, $m\ge 2$, non-compact, complete and connected Riemannian manifold. 
In what follows $B(x,r)$ will denote a geodesic ball in $M$ and
$\Omega_{r}$ will denote the ball in $\R^{m}$ of radius $r$ centered at the origin. For every $x\in M$ there exists a maximal $r(x)\in (0,\infty]$, called injectivity radius at point $x$, such that the Riemannian exponential map $\mathrm{exp}_{x}$ is a diffeomorphism of $\{v\in T_xM:\, |v|_x\eqdef \sqrt{g_x(v,v)}< r(x)\}$ onto $B(x,r(x))$. For each $x\in M$ we choose an orthonormal basis for $T_xM$ which yields an identification $i_x:\R^m \rightarrow T_xM$.  Then  $e_{x}:\Omega_{r}\to B(x,r(x))$ will denote  a geodesic normal coordinates at $x$ given by $e_{x}=\mathrm{exp}_{x}\circ i_x$. We do not require smoothness of the map $i_x$ with respect to $x$, since the arguments $x$ will be taken from a discrete subset of $M$.  We recall that $r(M)=\inf\{r(x): x\in M\}$ is called an injectivity radius of the manifold $M$. If $M$ is compact, $r(M)$ is always strictly positive, but it is not necessary so for non-compact manifolds. Since we  assume that $M$ is connected, the distance  $d_M(x,y)$ between any two points $x$ and $y$ on $M$ is well defined.

For $k$ integer, and $f:M\rightarrow \C$ we denote by $\nabla^k f$ the $k^{\text{th}}$ covariant derivative of $u$,  and by $|\nabla^k f|$ the norm of $\nabla^k f$ defined by a local chart by
\[
|\nabla^k f|^2 = g^{i_1j_1}\cdots g^{i_kj_k}\partial_{i_1}\ldots \partial_{i_k}  f\partial_{j_1}\ldots \partial_{j_k}  \overline{f}\, .
\] 
In what follows we assume the following conditions.
\begin{enumerate}
	\item[(M1)] The Ricci curvature of $M$ is bounded from below.
	\item[(M2)] $\inf_{x\in M}\mathrm{vol}B(x,1)>0$,
\end{enumerate}
\begin{remark}\label{rem:anyball}
	If (M1) holds, then it follows from Bishop-Gromov theorem (see \cite[Theorem 1.1]{Hebey}) that for any $0<r<R$ there is a $C(r,R)>0$ such that 
	\begin{equation}
	\label{eq:vol}
		\mathrm{vol}(B(x,R))\le C(r,R) \mathrm{vol}(B(y,r))
		\mbox{ for any } x\in M, y\in B(x,R).
	\end{equation}
	
	If (M1) and (M2) hold, then one has 
	\begin{equation}
		 \inf_{x\in M}\mathrm{vol}B(x,r)>0
	\end{equation}
	for any $r>0$.
\end{remark}
The Sobolev space $H^{1,p}(M)$, $p\in[1,\infty)$, is a completion of $C^{\infty}_o(M)$ with respect to the norm  
\[
\|u\|_{H^{1,p}}^p=  \int_M |\nabla u|^p d\mathrm{vol} + \int_M |u|^pd\mathrm{vol}.
\]
Let $p^*$ denote the  Sobolev conjugate of $p$, $1\le p < m$ i.e. $\frac{1}{p^*}=\frac{1}{p}- \frac{1}{m}$.

Since $M$ satisfies (M1) and (M2), the space  $H^{1,p}(M)$ is continuously embedded into $L^q(M)$ for every  $p\in(1,m)$ and $q\in[p,p^*]$ and the constant in Sobolev embeddings over balls $B(x,r)$ is independent of $x\in M$
(see \cite[Theorem 3.2 and Theorem 3.1]{Hebey} based on \cite{Buser}).

\begin{defn}
\label{def:discr}  A subset $\Gamma$ of  Riemannian manifold $M$ is called 
an $(\varepsilon,\nu)$-discretization of $M$, $\varepsilon>0$, $\nu\in \N$,  if the distance between any two distinct points of $\Gamma$ is greater than or equal to $\varepsilon$ and 
\[ M = \bigcup_{y\in \Gamma} B(y,\nu\varepsilon).\]
\end{defn}
Any Riemannian manifold $M$ has an $(\varepsilon,\nu)$-discretization for any $\varepsilon>0$ and $\nu \ge 1$. If $M$ satisfies (M1), then the covering  
$\{ B(y,r)\}_{y\in \Gamma}$ is uniformly locally finite for any $r\ge \nu\varepsilon$, cf. \cite[Lemma 1.1]{Hebey} and  \cite{GS}, \cite{Shubin}, \cite{LS} where the same concepts are considered with stronger assumptions about geometry.  

\begin{lemma}\label{lem:intheball}
Let $M$ satisfy (M1) and let $\Gamma$ be a $(\varepsilon,\nu)$-discretization of $M$. Then for any $R>0$ there exists $n_R\in\N$, such that  $\#(\Gamma\cap B(x,R))\le n_R$ for every $x\in M$.
\end{lemma}
\begin{proof}
By definition,  $\#(\Gamma\cap B(x,R))$ cannot exceed the maximal number of disjoint balls of radius $\varepsilon/2$ contained in $B(x,R+\varepsilon)$, which is finite by \eqref{eq:vol}. 
\end{proof}

\begin{lemma}["Spotlight lemma"]
	\label{lem:spotlight}Let $M$ be an $m$-dimensional,  non-compact, complete  Riemannian manifold satisfying (M1) - (M2),
	and let $\Gamma\subset M$ be a $(\varepsilon,\nu)$-discretization of $M$, 
	$\varepsilon,\nu>0$.
	Let  $(u_{k})$ be a bounded sequence in $H^{1,p}(M)$, $1< p< m$.  Then, $u_{k} \to 0$ in $L^{q}(M)$ for any $ q\in(p,p^{*})$ if and only if  
	\begin{equation}
	\label{eq:spotlight}
	\int_{B(y_k,\nu\varepsilon)}|u_{k}|d\mathrm{vol}\to 0 \mbox { for any sequence }(y_k),\quad y_{k}\in \Gamma.
	\end{equation}
\end{lemma}

\begin{proof}  
	
	Necessity in the lemma is  trivial. 
	
	Let us prove sufficiency. Assume condition \eqref{eq:spotlight}.   The  Sobolev inequalities on Riemannian balls, cf. eg. \cite{MSC95},  implies   that there  exists a positive constant 
	$C>0$ independent of $y\in M$ such that 
	\[
	\int_{B(y,\nu\varepsilon)}|u_{k}|^{q}d\mathrm{vol}\le C\int_{B(y,\nu\varepsilon)}(|\nabla u_{k}|^{p}+|u_{k}|^{p})d\mathrm{vol}\left(\int_{B(y,\nu\varepsilon)}|u_{k}|^{q}d\mathrm{vol}\right)^{1-p/q}.
	\]
	Adding the terms in  the left and  the right hand side over $y\in \Gamma$ 
	and taking into account the uniform multiplicity of the covering (a consequence of (M1)),
	we have 
	\begin{equation}
	\int_{M}|u_{k}|^{q}d\mathrm{vol}\le C\int_{M}(|\nabla u_{k}|^{p}+|u_{k}|^{p})d\mathrm{vol}\; \sup_{y\in\Gamma}\left(\int_{B(y,\nu\varepsilon)}|u_{n}|^{q}d\mathrm{vol}\right)^{1-p/q}.\label{eq:intro2}
	\end{equation}
	Boundedness of the sequence  $(u_{k})$  in $H^{1,p}(M)$ implies that the supremum of the right hand side is finite. So for any $u_{k}$, $k\in\N$, we can find a 
	$y_{k}\in \Gamma$, 
	such that 
	\begin{equation}
	\label{eq:intro1}
	\sup_{y\in \Gamma}\int_{B(y,\nu\varepsilon)}|u_{k}|^{q}d\mathrm{vol} \le 2\int_{B(y_{k},\nu\varepsilon)}|u_{k}|^{q}d\mathrm{vol}.
	\end{equation}
	Applying the H\"older inequality to the right hand side, we have
	\begin{equation}
	\label{eq:introX}
		\int_{B(y_{k},\nu\varepsilon)}|u_{k}|^{q}d\mathrm{vol}\le 
	\|u_k\|_{p^*}^\frac{p^*(q-1)}{p^*-1}	\left( \int_{B(y_{k},\nu\varepsilon)}|u_{k}|d\mathrm{vol}
		\right)^\frac{p^*-q}{p^*-1},
	\end{equation}
	which, given that $(u_k)$ is bounded in $L^{p^*}(M)$, converges to zero by \eqref{eq:spotlight}. Combining \eqref{eq:intro2}, \eqref{eq:intro1}, and \eqref{eq:introX} we have $u_k\to 0$ in $L^q(M)$.
\end{proof} 

As a consequence of the spotlight lemma we have the following compactness property for functions supported on sets thin at infinity. For an open set $M_0$ of a Riemannian manifold $M$ we denote the closure of the space of Lipschitz functions with compact support on $M_0$ in the norm of $H^{1,p}(M)$ as $H^{1,p}_0(M_0)$. We will call a sequence $(y_k)$ in $M$  discrete if it contains no bounded subsequence. 

\begin{proposition}
	Let $M$ be an $m$-dimensional non-compact, complete Riemannian manifold satisfying conditions (M1)-(M2), let $M_0$ be an open subset of $M$,  and let $1<p<m$. Let $\Gamma\subset M$ be a $(\varepsilon,\nu)$-discretization of $M$, $\varepsilon,\nu>0$. 
	If for any discrete sequence $(y_k)$ in $\Gamma$
	\begin{equation} \label{eq:vol0}
	\mathrm{vol}(M_0\cap B(y_k,\nu\varepsilon))\to 0,	
	\end{equation}
	then $H^{1,p}_0(M_0)$ is compactly embedded into $L^q(M_0)$, $p<q<p^*$.
\end{proposition}
\begin{proof} Let $(u_k)$ be a sequence in  $H^{1,p}_0(M_0)$, weakly convergent to zero.
Then by compactness of local Sobolev embeddings, for any $y\in \Gamma$, 
$\int_{B(y,\nu\varepsilon)}| u_k|d\mathrm{vol}\to 0$.
If $(y_k)$, $ y_k\in\Gamma $, is a bounded sequence then it consists of finitely many values since $\Gamma$ is a discretization. In consequence  
\begin{equation}
\label{eq:dsq}
\int_{B(y_k,\nu\varepsilon)}| u_k|d\mathrm{vol}\to 0 \mbox{ for any bounded sequence } (y_k),  y_k\in\Gamma.
\end{equation}
On the other hand, if  $(y_k)$ be an aribitrary discrete sequence in $\Gamma$, 
by H\"older inequality and \eqref{eq:vol0},
 \begin{align*}
\int_{B(y_k,\nu\varepsilon)}| u_k|d\mathrm{vol}\le
\\
\left(\int_{B(y_k,\nu\varepsilon)}| u_k|^{p^*}d\mathrm{vol}\right)^{1/p^*}
\Big( \mathrm{vol}(M_0\cap B(y_k,m\varepsilon))\Big)^{1-1/p^*}\le 
\\
C\|u_k\|_{H^{1,p}(M)}\Big( \mathrm{vol}(M_0\cap B(y_k,\nu\varepsilon))\Big)^{1-1/p^*}
\to 0.
\end{align*}
Combining this with \eqref{eq:dsq} we have \eqref{eq:spotlight}. Then by Lemma~\ref{lem:spotlight} $u_k\to 0$ in $L^q(M)$, which proves the proposition.
\end{proof}	
\section{Orbital discretization and general compactness theorem}
\begin{definition}
	An $(\varepsilon,\nu)$- discretization $\Gamma$ of a Riemannian manifold $M$ is called an orbital discretization if there exist nonempty subsets $\Gamma_i\subset \Gamma$, $i\in \N$, such that
	\begin{enumerate}
		\item[(a)]  $\Gamma = \bigcup_{i=1}^\infty\Gamma_i$ and 
		  $\Gamma_i\cap \Gamma_j= \emptyset$ if $i\not=j$,
		\item[(b)] $\#\Gamma_i\le \#\Gamma_{i+1}<\infty$, $i\in\N$,
		\item[(c)] $\lim_{i\to\infty}\#\Gamma_i=\infty$.
 	\end{enumerate}
We shall write then $\Gamma\in\mathcal O_{\varepsilon,\nu}(M)$. The sets $\Gamma_i$ will be called quasiorbits.  
\end{definition}
The term {\em orbital discretization} will be  justified in the next subsection when we discretize  group orbits on a manifold.   

\begin{lemma}\label{l:balls}
	Let $\Gamma$ be an orbital discretization. For every $R>0$ and $j\in\N$ there exists $\bar i(R,j)\in \N$ such that for all $i\ge \bar i(R,j)$ and for every $x\in \Gamma_i$, there exists a subset $\Gamma_i(x)\subset\Gamma_i$ satisfying
	\begin{enumerate}
		\item[(i)] $x\in\Gamma_i(x)$,
		\item[(ii)] $d(y,z)>R$ whenever $y,z\in \Gamma_i(x)$, $y\neq z$,
		\item[(iii)] $\#\Gamma_i(x)\ge j$.
	\end{enumerate}	 
\end{lemma}
\begin{proof} For $j=1$ conditions (i - iii) hold tautologically when $\Gamma_i(x)=\{x\}$. 
	 We assume now that $j\ge 2$.
Let $n_R$ be as in Lemma~\ref{lem:intheball} and let $i_0\in\N$ be such that $\#\Gamma_i>jn_{R}$ for any $i\ge i_0$. Such $i_0$ always exists by property (c) in the definition of the orbital discretization. Let $y_0=x$ and let us choose recursively
$y_{k+1}\in\Gamma_i$, $k=0,\dots,j-2$, such that 
$y_{k+1}\notin B(y_\ell,R)$, $\ell=0,\dots,k$. This is possible since the balls $B(y_\ell,R)$, $\ell=0,\dots,k$ contain all together not more than  $(k+1)n_R$ points of $\Gamma_i$, and this number is less than $jn_R$ and thus less than $\#\Gamma_i$. Obviously, $d(y_k,y_\ell)>R$ whenever $k\neq \ell$. We set $\Gamma_i(x)=\{y_k\}_{k=0,\dots,j-1}$.
\end{proof}
\begin{corollary} 
Let $\Gamma$ be an orbital discretization. Then $\lim_{i\to\infty}\mathrm{diam}\, \Gamma_i=\infty$.	
\end{corollary}

\begin{definition}
Let $\Gamma\in\mathcal O_{\varepsilon,\nu}(M)$, $\nu\varepsilon<r(M)$. Let $i\in \N$ and $\lambda\ge 1$. A function $f\in L^1_{loc}(M)$ is called $(i,\lambda)$-quasisymmetric relative to $\Gamma$  if for every $\ell\ge i$ 
\begin{equation}\label{eq:qsym}
\max_{x\in\Gamma_{\ell}} \int_{B(x,\nu\varepsilon)}|f(y)| d\mathrm{vol} \, \le \, \lambda\, \min_{x\in\Gamma_{\ell}} \int_{B(x,\nu\varepsilon)}|f(y)| d\mathrm{vol}  .\end{equation} 
We shall  write then $f\in \mathcal{S}_{\Gamma,i,\lambda}(M)$.  
\end{definition}
\begin{theorem}\label{t:main}
Let $M$ be complete, noncompact, connected, $m$-dimensional Riemannian manifold satisfying (M1) - (M2). Let $\Gamma\in\mathcal O_{\varepsilon,\nu}(M)$. Let $1<p<m= \dim M$,  $p<q<p^*$, $i\in \N$ and  $\lambda\ge 1$. If a set $K\subset H^{1,p}(M)\cap \mathcal{S}_{\Gamma,i,\lambda}(M)$ is bounded in  $H^{1,p}(M)$  then it is relatively compact in $L^q(M)$.
\end{theorem}

\begin{remark}
	1. For any $\Gamma$, $i$ and $\lambda$ the set $\mathcal{S}_{\Gamma,i,\lambda}(M)$ contains infinitely many linearly independent functions from $H^{1,p}(M)$. In particular, it has the following functions. Let $\varphi_x\in C^\infty(M)$ be a nonnegative nonzero function with $\mathrm{supp}\, \varphi_x\subset B(x,\varepsilon/2)$, $x\in \Gamma_\ell$, and define 
	\[ f(y)= \sum_{x\in \Gamma_{\ell}} \frac{\varphi_x(y)}{\int_M \varphi_x d \mathrm{vol}},\, \qquad \ell\ge i.\]   
	2.   For any $\Gamma$, $i$, and $\lambda$ the set $H^{1,p}(M)\cap \mathcal{S}_{\Gamma,i,\lambda}(M)$ is closed with respect to the weak convergence in  $H^{1,p}(M)$, since all the quantities in the relation \eqref{eq:qsym} are weakly continuous in  $H^{1,p}(M)$. 
\end{remark}

\begin{proof} By reflexivity
it is sufficient to show that if $(u_k)$ is a sequence in $ H^{1,p}(M)\cap \mathcal{S}_{\Gamma,i,\lambda}(M)$ weakly convergent to zero in  $H^{1,p}(M)$ then  $u_k\to 0$ in $L^q(M)$.
Assume that this is not the case. Then by Lemma~\ref{lem:spotlight} there is a sequence $(y_k)$, $y_k\in \Gamma$, and $\delta>0$ such that 
\begin{equation}
\label{eq:bottom}
\int_{B(y_k,\nu\varepsilon)}|u_k| d\mathrm{vol} \,\ge\delta.
\end{equation}

Note that if the sequence $(y_k)$ has a bounded subsequence it has the constant subsequence, by compactness
of local Sobolev embeddings \eqref{eq:bottom} cannot hold, and thus $(y_k)$ is necessarily discrete.
 So we can assume that $y_k\in \Gamma_{\ell_k}$ with $\ell_k>i$ and $\ell_k\rightarrow \infty$. The functions $u_k$ are of the quasisymmetry class $\mathcal{S}_{\Gamma,i,\lambda}(M)$, so  by  \eqref{eq:qsym}, for $k$ large enough we have for every $x\in \Gamma_{\ell_k}$, $\ell_k\ge i$,
\begin{align}
\int_{B(x,\nu\varepsilon)}|u_k| d\mathrm{vol} \,\ge\, C\lambda\int_{B(y_k,\nu\varepsilon)}|u_k| d\mathrm{vol}\ge \, C \lambda\delta\eqdef\beta > 0 .  
\end{align}
Let us apply Lemma~\ref{l:balls} with $R=2\nu\varepsilon$ and for each $j\in\N$  choose $k_j$ such that
 $\ell_{k_j}\ge \bar i(2\nu\varepsilon,j)$. This gives 
\begin{equation}
	\int_M |u_k|^q d\mathrm{vol} \ge  C_{\nu\varepsilon} \sum_{x\in \Gamma_{\ell_{k_j}}}\int_{B(x,\nu\varepsilon)} |u_k| d\mathrm{vol}  \ge  C_{\nu\varepsilon} j\beta.
\end{equation}  
Since $j$ is arbitrarily large, we have a contradiction that proves the theorem. 
\end{proof}
%

\begin{example}
Let $M=\R^m$ be equipped by the usual Euclidean metric. Let $A$ be a real $m\times m$ matrix with eigenvalues $\lambda_j$, $\Re \lambda_j > 0$. We set $ \lambda = \min_{ 1\le j \le m} \Re \lambda_j$ and $ \Lambda = \max_{ 1\le j \le m} \Re \lambda_j$. Following Stein and Wainger \cite{SW} associate to
$A$ the dilation matrix $\delta_t= \exp(A\ln t)x$. Moreover, we can introduce a positive, $\delta_t$-homogeneous distance functions $\varrho$, i.e., a continuous functions $\varrho$ on
$\R^m$  such that
\begin{align*} 
 \varrho(x)\ge 0\qquad \text{and}\quad  	\varrho(x)>0\; \text{for} \; x\not=0,\\
 \varrho(\delta_t(x)) = t \varrho(x)\; \text{for all}\;  t>0,\;  x\in \R^m.
 \end{align*}

Furthermore, one can prove that  for any $\eta>0$ there are positive constants $c_1,c_2$ such that
\begin{align*}
c_1 |x|^{1/(\lambda-\eta)} \le \varrho(x) \le c_2 |x|^{1/(\Lambda+\eta)}\qquad \text{if} \quad |x| >1 ,\\
c_1 |x|^{1/(\Lambda+\eta)} \le \varrho(x) \le c_2 |x|^{1/(\lambda-\eta)}\qquad \text{if} \quad |x|< 1 .  
\end{align*} 
The finctions of the form $x\mapsto f(\varrho(x))$ are called quasiradial. 
The sets $\Sigma_\varrho (r) =\{x\in \R^m: \varrho(x)=r \}$ are compact. One can  easily  construct an orbital $(\varepsilon, \nu)$-discretization of $\R^m$ such that any quasiorbits $\Gamma_i$ is contained in some of the set  $\Sigma_\varrho (r)$ and different quasiorbits are contained in   different sets  $\Sigma_\varrho (r)$. Theorem \ref{t:main} implies that the subspace of quasiradial functions in $H^{1,p}(\R^m)$ is compactly embedded into $L^q(\R^m)$. 
\end{example}
\section{Compactness for functions with group symmetry}
Any discretization of a noncompact manifold can be partitioned as an orbital discretization. However, when one wants, as in this section, to study compactness of embedding of  spaces invariant with respect to a group action it is natural to consider a specific kind of orbital discretizations, namely those associated with the group orbits. Similarly, in the next section we will study 
compactness of embedding of  spaces with reduced number of variables, where quasiorbits are associated with the level sets of a map.  
 
Let $G$ be a compact connected  group of isometries of the manifold $M$. Then $H^{1,p}_G(M)$ will denote a subspace of $H^{1,p}(M)$ consisted of all $G$-invariant functions.  We will use the notion  of coercive group action introduced in \cite{Skrti1}.
\begin{definition}\label{def:coerciveG}
	We say that a continuous action of a group $G$ on a
	complete Riemannian manifold $M$ is {\em coercive} if for every $t>0$,
	the set
	\[
	O_t=\lbrace x\in M: \; \mathrm{diam}\, G x\le t\rbrace
	\]
	is bounded.
\end{definition}
If the sectional curvature of $M$ is non-positive and the compact connected group $G$ of isometries fixes some point, then $G$ is coercive if and only if  $G$ has no other fixed point: see \cite[Proposition 3.1]{Skrti1}. An example of a compact connected coercive group without fixed points (see the end of \cite[Section 3]{Skrti1}) is
$M=S^1\times \R^n$ (a Riemannian product of the unit
circle and the Euclidean space), $n\ge 2$, and $G = S^1\times SO(n)$ acting on $M$
by the formulae $(e^{i\varphi},h)(e^{i\psi},x) = (e^{i(\varphi+\psi)},h(x))$,
$e^{i\varphi},e^{i\psi}\in S^1$, $h\in SO(n)$ and $x\in \R^n$.
 
\begin{proposition}\label{p:group}
	Let $G$ be a compact, connected group of isometries acting coercively on the manifold $M$. Then there exists an orbital discretization $\Gamma\in \mathcal O_{\varepsilon,2}(M)$ such that  any  quasiorbit $\Gamma_i$  is a subset of a distinct orbit of $G$.  
\end{proposition}

\begin{proof}
Let $\widetilde{M}$ be a union of all principal orbits of the group $G$. The set $\widetilde{M}$	is a dense open subset of $M$, cf. \cite[Chapter IV, Theorem 3.1]{Bre}. On the coset space $\widetilde{M}/G$ one can introduce a Riemannian structure such that the projections $p: \widetilde{M} \rightarrow \widetilde{M}/G$  have the following property
\[ d_{\widetilde{M}/G}(p(x),p(y)) = d_M (Gx, Gy) \]
where the  distances are taken on respective manifolds, cf. \cite[Theorems 2.28 and 2.109]{GHL}. Let $\widetilde{\Gamma}=\{Gx_\ell\}_{\ell\in \N}$ be an $(\varepsilon,1)$- discretization of $\widetilde{M}/G$ with $\varepsilon< r(M)/3$.  Let $\dot{\Gamma}_\ell$ be an $(\varepsilon, 1)$-discretization of the orbit $Gx_\ell$ in $M$. Then $\Gamma=\bigcup_{\ell=1}^\infty \dot{\Gamma}_\ell$ is an $(\varepsilon,2)$-discretization of $M$. Let $\{\Gamma_i\}$ be the family $\{\dot{\Gamma}_\ell\}$ reordered by the number of elements in $\dot{\Gamma}_\ell$. Then  $\Gamma=\bigcup_i\Gamma_i$ is obviously a $(\varepsilon,2)$-discretization of $M$. We  prove that it is an orbital discerization.  Conditions (a) and (b) are satisfied by the construction.  The condition (c) is a consequence of the coercivity of the action of $G$ as follows. Let $R>0$. By the coercivity all sets $\Gamma_i$ of diameter not exceeding $R$ lie in a bounded set $O_R$. However, only finitely many elements of $\Gamma$ may lie in $O_R$. So there exists $i_R\in \N$ such that diameter of $Gx_\ell$ is greater then $R$ whenever $\ell\ge i_R$. The orbits  $Gx_\ell$ are connected since $G$ is connected, therefore $\#\Gamma_{\ell}\rightarrow \infty$.    
\end{proof}

Taking into account the above proposition one can apply Theorem \ref{t:main} to sets of quasisymmetric functions related to the action of a group $G$ of isometries of $M$. In particular it can be applied to the subspaces $H^{1,p}_G(M)$  of $H^{1,p}(M)$ consisting of all $G$-symmetric functions.   

\begin{theorem}\label{thm:groupsymm}
	Let $G$ be a compact, connected group of isometries of an $m$-dimensional, non-compact, connected and complete Riemannian manifold $M$ satisfying (M1) - (M2). Let $1<p<m$ and $p<q<p^*$. If  $G$ is coercive then the subspace  $H^{1,p}_G(M)$ is compactly embedded into $L^q(M)$.
	Furthermore, if $H^{1,p}_G(M)$ is compactly embedded into $L^q(M)$ then $G$ is coercive
	provided that the injectivity radius $r(M)=\inf_{x\in M}r(x)$ of the manifold $M$ is positive.
\end{theorem}

\begin{proof}
	Sufficiency in the theorem follows from Theorem~\ref{t:main} with the orbital discretization given by Proposition \ref{p:group} since  $H^{1,p}_G(M)\subset H^{1,p}(M)\cap \mathcal{S}_{\Gamma,1,1}$. In particular, by isometry, 
	\[
	\int_{B(x,\nu\varepsilon)}|f(y)| d\mathrm{vol} \, = \,  \int_{B(z,\nu\varepsilon)}|f(y)| d\mathrm{vol}, \qquad z\in Gx
	 \]
	whenever $f\in H^{1,p}_G(M)$. 
	
	{\em Proof of necessity}. If $G$ is not coercive, there exists $R>0$ and a discrete sequence $(x_k)$ in $M$ such that $G x_k\subset B(x_k,R)$. 
	Let $r\in(0,r(M))$ and let us replace $x_k$ with a renumbered subsequence such that distance between any two terms in the sequence will be greater than $2(R+r)$. 
	Let
	$$
	\psi_k(x)\eqdef\int_G [r-d_M(gx, x_k)]_+\; \mathrm{d}g,\:x\in M,
	$$
	where the Haar measure of $G$ is normalized to the value $1$. By the Minkowski integral inequality, taking into account that $G$ is a group of isometries on $M$ and that the injectivity radius of $M$) is positive, we have
	\begin{align*}
		\|\psi_k\|_{H^{1,p}(M)}\le \int_G \| [r-d_M(g\cdot,x_k)]_+\|_{H^{1,p}(M)} \mathrm{d}g=
		\\
		\int_G \| [r-d_M(\cdot ,x_k)]_+\|_{H^{1,p}(M)} \mathrm{d}g=
		\\
		\| [r-d_M(\cdot ,x_k)]_+\|_{H^{1,p}(M)}
		\le
		C.
	\end{align*}
The constant $C$ is independent of $k$, since, using the normal coordinates at $x_k$ one has $|\nabla d_M(x,x_k)|_g=1$, $x\neq x_k$. 
	Note that the supports of the functions $\psi_k$ are disjoint, and therefore 
	$$
	\|\psi_\ell-\psi_n\|^q_{L^q(M)}=\|\psi_m\|^q_{L^q(M)}+\|\psi_n\|^q_{L^q(M)}\ge2\inf_k\|\psi_k\|^q_{L^q}.
	$$
	Furthermore, 
	\begin{align*}
		\mathrm{vol}(B({x_k,R+r}))^{1-1/q}&\|\psi_k\|_{L^q}
		\ge \int_M \psi_k\,d\mathrm{vol}=
		\\
		&\int_G \int_M[r-d_M(g\cdot ,x_k)]_+\;d\mathrm{vol}\;\mathrm{d}g
		=
		\\
		&\int_M[r-d_M(\cdot , x_k)]_+\;d\mathrm{vol}\ge  \frac12\mathrm{vol}(B(x_k,r/2).
	\end{align*}
Since, by  (M1)-(M2), $\sup_{k\in\N}\mathrm{vol}(B({x_k,R+r})<\infty$ and \\ $\inf_{k\in\N}\mathrm{vol}(B(x_k,r/2)>0$, $\|\psi_k\|_{L^q(M)}$ is bounded away from zero. Therefore 
	we have a sequence, bounded in $H^{1,p}(M)$ and discrete in $L^q$, and so
	the embedding $H^{1,p}(M)\hookrightarrow L^q(M)$ is not compact.
\end{proof}
\section{Compactness for functions with reduced number of variables}
In this section we will study compactness caused by reduction of the number of variables, i.e. compactness of subspaces of functions of the form $f\circ\varphi$ with fixed $\varphi$, for example $f:\R\to\R$ and 
$\varphi:\R^m\to\R$ defined by $\varphi(x)=|x|$. 
 
 We assume that $M$ is a complete smooth connected $m$-dimensional non-compact Riemannian manifold and  $N$ is a smooth $n$-dimensional connected Riemannian manifold, $n<m$. 
 Let $\varphi:M\to N$ be a Lipschitz-continuous map, which implies it is differentiable almost everywhere on $M$.
 We will use the classical coarea formula relative to $\varphi$, cf. \cite{Fed}:  for any measurable non-negative function $u(x)$,
 \begin{equation}
 \label{eq:coarea}
 \int_M u(x) \mathcal J_\varphi(x) d\mathrm{vol}_M(x)=\int_N\left[\int_{\varphi^{-1}(z)}u(x)\mathrm d\mathcal H_{m-n}(x) \right] d\mathrm{vol}_N(z),
 \end{equation}
 where 
$\mathcal J_\varphi(x)$
is the absolute value of the  normal Jacobian  of $\varphi$  (the determinant of the pushforward $d_x\varphi$ restricted to the orthogonal complement to its kernel) on 
the level set $\varphi^{-1}(z)$, $z\in N$, and $\mathcal H_{m-n}$ is the $m-n$-dimensional Hausdorff measure on $\varphi^{-1}(z)$. 
If, additionally, $\varphi\in C^{m-n+1}(M,N)$, then, by Sard's theorem, almost every $z\in N$ is a regular value of $\varphi$ and for every such $z$ the set
 $\varphi^{-1}(z)\subset M$ has a natural structure of $m-n$-dimensional Riemannian manifold, whose $m-n$-dimensional Hausdorff measure becomes
 the Riemannian measure on $\varphi^{-1}(z)$ with the Riemannian structure inherited from $M$ (see \cite[page 159]{Chavel}   for details).
 
Applying the coarea formula to the characteristic function of the set $\{x: \mathcal{J}_\varphi(x)=0\}$ we discover that 
\[ \mathcal{H}_{m-n}(\{x: \mathcal{J}_\varphi(x)=0\}\cap \varphi^{-1}(z))=0  \] 
for $\mathrm{vol}_N$ a.e. $z\in N$. This is a weak variant of Sard's theorem that holds for Lipschitz mappings. Thus the function $x\mapsto \mathcal{J}_\varphi(x)^{-1}$ is $\mathcal{H}_{m-n}$ a.e. finite on the level set $\varphi^{-1}(z)$ for  $\mathrm{vol}_N$ a.e. $z\in N$. 
 We assume that
 \begin{equation}
 \label{eq:Jphinz}
 \mathcal J_\varphi\neq 0\mbox{ a. e. on }M,
 \end{equation}
 and
\begin{equation}
\label{eq:mualpha}
\Psi(z)\eqdef \int_{\varphi^{-1}(z)}\frac{\mathrm{d}\mathcal H_{m-n}}{\mathcal J_\varphi}\quad<\infty\text{ for a.e. }z\in N.
\end{equation}

Note that set $\varphi(M)$ is connected. 
Consider the following subspace of $H^{1,p}(M)$:
\begin{align}
\label{eq:Sob-phi}
H^{1,p}_\varphi(M) \eqdef \{g\in H^{1,p}(M): &\; g = f\circ\varphi \quad \text{with}\quad \\
& f\in  L^p(\varphi(M), \Psi d\mathrm{vol}_{N})\}. \nonumber
\end{align}
By the coarea formula applied to $\frac{|f_k\circ\varphi|^p}{\mathcal J_\varphi}$, functions in $H^{1,p}_\varphi(M)$ satisfy the following relation:
\begin{equation}
\label{eq:Lq}
\int_M |f\circ\varphi|^p d\mathrm{vol}_M(x)=\int_{\varphi(M)} |f(z)|^p \Psi(z) d\mathrm{vol}_N(z).
\end{equation}
\begin{proposition}
If $\varphi$ is Lipschitz-continuous and satisfies \eqref{eq:Jphinz} and \eqref{eq:mualpha}, then $H^{1,p}_\varphi(M)$ is a closed subspace of $H^{1,p}(M)$. 
Moreover, if $\varphi(M)$ has a nonempty interior with and a boundary of measure zero, and $f$ is a function on $N$ such that both functions  $f$  and $\nabla_N f$ are in  $L^p(\varphi(M), \Psi d\mathrm{vol}_{N})$, then $f\circ\varphi\in H^{1,p}_\varphi(M)$. 
\end{proposition}
\begin{proof}
If  a sequence $f_k\circ\varphi$ converges in $H^{1,p}(M)$, then by \eqref{eq:Lq} $f_k$ converges in $L^p(\varphi(M), \Psi d\mathrm{vol}_{N})$ to some $f$ and, also by \eqref{eq:Lq},  $f_k\circ\varphi$ converges to  $f\circ\varphi$ in $L^p(M)$. Thus the $H^{1,p}(M)$-limit of $f_k\circ\varphi$ is $f\circ\varphi$.
This proves the first assertion of the proposition. To verify the second assertion, note that, by the chain rule (applied under our assumption on $\varphi(M)$) we have a relation for $\nabla (f\circ\varphi)$, similar to \eqref{eq:Lq}, namely
\begin{align*}
\label{eq:nablaLq}
\int_M |\nabla_M(f\circ\varphi)|_M^p d\mathrm{vol}_M(x)\le C 
 \int_M |\nabla_Nf|_N^p\circ\varphi \;d\mathrm{vol}_M(x)=
 \\
 C\int_{\varphi(M)} |\nabla_Nf(z)|_N^p \Psi(z) d\mathrm{vol}_N(z)<\infty.
 \end{align*}	
\end{proof}
In what follows we will denote by $B$  balls in $M$ and by $B_N$ balls in $N$. 
For $r>0$ and an open set $A\subset M$ define
\begin{equation}
\label{eq:deltaalpha}
\delta_r(A)\eqdef\sup_{y\in M\setminus A,z\in \varphi(M)}
\dfrac{\int_{\varphi^{-1}(z)\cap B(y,r)}\mathcal J_\varphi(x)^{-1}\mathrm d\mathcal H_{m-n}(x)  }{\int_{\varphi^{-1}(z)}\mathcal J_\varphi(x)^{-1}\mathrm d\mathcal H_{m-n}(x)  }.
\end{equation}
Quantity $\delta_r(A)$ relates the volume of the portion of a level set $\varphi^{-1}(z)$ inside a small geodesic ball centered outside of a given set $A$, on one hand, to the volume of the whole level set, on the other. By definition it is monotone nonincreasing with respect to $A$. Below we connect compactness of Sobolev embeddings  to  $\delta_r(A)$ vanishing at infinity. 
\begin{definition} 
\label{def:delta-vanish}
Let $r>0$.
We shall say that the quantity \eqref{eq:deltaalpha} vanishes at infinity if there exists a countable exhaustion (i.e. monotone covering) $\{A_k\}_{k\in\N}$ of $M$ by open bounded sets such that $\delta_r(A_k)\to 0$ if $k\rightarrow \infty$.  
\end{definition}

\begin{lemma}
If $\delta_r$ vanishes at infinity then, for {\em any} countable exhaustion  $\{A_k\}_{k\in\N}$ of $M$ by open bounded sets, $\delta_r(A_k)\to 0$ if $k\rightarrow \infty$.
\end{lemma}
\begin{proof}
Let   $\{A_k\}_{k\in\N}$ be as in Definition~\ref{def:delta-vanish} and let $\{A'_k\}_{k\in\N}$ be another exhaustion of $M$ by open bounded sets. Since $\overline {A_k}$ is a compact set, it is covered by finitely many sets $A'_n$, and since the latter family is monotone, it is covered by some single set $A'_{n_k}$. By monotonicity of $\delta_r$ we have $\delta_r(A'_{n_k})\le \delta_r(A_{k})\to 0$. Note that the sequence $(n_k)$ is unbounded, since otherwise $M=\cup_{k\in\N} A_k$ would be contained in a bounded set. Then by monotonicity of  $\{A'_k\}_{k\in\N}$  and of $\delta_r$ we have $\delta_r(A'_{k})\to 0$.
\end{proof}
\begin{theorem}\label{thm:levels1}
 Let $M$ be a complete  connected non-compact smooth Riemannian $m$-dimensional manifold satisfying (M1)-(M2), let $N$ be a smooth $n$-dimensional Riemannian manifold, $n<m$. Let 
$\varphi:M\to N$ be a  Lipschitz map 
satisfying \eqref{eq:Jphinz} and \eqref{eq:mualpha}.

If, for some $r>0$, the quantity $\delta_r$ vanishes at infinity,  
then the subspace  $H^{1,p}_\varphi(M)$,  $p\in(1,m)$,  is compactly embedded into $L^q(M)$ for every 
$q\in(p,p^*)$.
\end{theorem}
\begin{proof}
Assume that $u_k=f_k\circ\varphi\rightharpoonup u$ in $H^{1,p}_\varphi(M)$.
Let us fix $R>0$.	
Let $\Gamma$ be a $(\epsilon,r)$-discretization of $M$.  Similarly to the proof of Lemma~\ref{lem:spotlight},
\begin{align}\label{23.01-1}
\int_{B(y,r)}|u_{k}|^{q}& d\mathrm{vol}_M\le \\ \nonumber
&C\int_{B(y,r)}(|\nabla u_{k}|^{p}+|u_{k}|^{p})d\mathrm{vol}_M\left(\int_{B(y,r)}|u_{k}|^{q}d\mathrm{vol}_M\right)^{1-p/q}.
\end{align}
Let $\{A_\ell\}_{\ell\in\N}$ be a monotone covering of $M$ by 
  bounded domains with Lipschitz boundary and let
$\tilde A_\ell\eqdef\{x\in M: d(x,A_\ell)<r\}$. 

Adding the inequalities  \eqref{23.01-1} over all $y\in \Gamma\setminus A_\ell$, we have
\[
\int_{M\setminus \tilde A_\ell}|u_{k}|^{q}d\mathrm{vol}_M\le C \|u_k\|^p_{H^{1,p}(M)}
\left(\sup_{y\in \Gamma\setminus A_\ell}\int_{B(y,r)}|u_{k}|^{q}d\mathrm{vol}_M\right)^{1-p/q}.
\]	
Then, using \eqref{eq:coarea} and \eqref{eq:deltaalpha}, we get
\begin{align}
\int_{B(y,r)} & |u_{k}|^{q} d\mathrm{vol}_M = \\  
&\int_N\left[\int_{x\in \varphi^{-1}(z) \cap B(y,r) }|u_k(x)|^q\mathcal J_\varphi(x)^{-1}\mathrm d\mathcal H_{m-n}(x)  \right] d\mathrm{vol}_N(z) = \nonumber 
\end{align} 
\begin{align}
&\int_N |f_k(z)|^q\left[\int_{x\in \varphi^{-1}(z) \cap B(y,r) }\mathcal J_\varphi(x)^{-1}\mathrm d\mathcal H_{m-n}(x)  \right] d\mathrm{vol}_N(z) \le   \nonumber \\
& \delta_r(A_\ell)\int_N |f_k(z)|^q \left[\int_{x\in \varphi^{-1}(z)}\mathcal J_\varphi(x)^{-1}\mathrm d\mathcal H_{m-n}(x) \right] d\mathrm{vol}_N(z)  =  \nonumber \\
& \delta_r(A_\ell) \int_{M}|u_{k}|^{q}  d\mathrm{vol}_M \nonumber .
\end{align}
Now taking into account that $(u_k)$ is a bounded sequence in $H^{1,p}(M)$ and in consequence in any space $L^q(M)$ for $p\le q\le p^*$ 
we have
\begin{equation}
\int_{M\setminus \tilde A_\ell}|u_{k}|^{q}d\mathrm{vol}_M\le C\delta_r(A_\ell)^{1-p/q}.
\end{equation}
Passing to the weak limit and using weak semicontinuity of norms, we have the same estimate  for $u$,
and therefore, 
\begin{equation}
\int_{M\setminus \tilde A_\ell}|u_{k}-u|^{q}d\mathrm{vol}_M\le 2^{q-1}C\delta_r(A_\ell)^{1-p/q}.
\end{equation}
Let $H^{1,p}(\tilde{A}_\ell)$ be a Sobolev space on $\tilde{A}_\ell$ defined by restrictions. The  domain $\tilde{A}_\ell$ is bounded therefore the Sobolev embedding   $H^{1,p}(\tilde{A}_\ell)\hookrightarrow L^q(\tilde{A}_\ell)$ is compact. Moreover,  $u_k-u\rightharpoonup 0$ in $H^{1,p}(M)$ therefore we have
\begin{align*}
\limsup_{k\to\infty}\int_{M}|u_{k}-u|^{q}d\mathrm{vol}_M\le 
\limsup_{k\to\infty}\int_{M\setminus \tilde A_\ell}|u_{k}-u|^{q}d\mathrm{vol}_M
+
\\
\limsup_{k\to\infty}\int_{\tilde A_\ell}|u_{k}-u|^{q}d\mathrm{vol}_M
\le 2^{q-1}C\delta_r(A_\ell)^{1-p/q}
\end{align*} 
Since $\delta_r$ vanishes at infinity, by taking $\ell\to\infty$ we arrive at $u_k\to u$ in $L^q(M)$.
\end{proof}
Based on the example of equivalence of spaces $H_{O(m)}^{1,p}(\R^m)$ and  $H_{\varphi}^{1,p}(\R^m)$ with $\varphi(x)=|x|$, it could be natural to introduce a coercivity property of the map $\varphi$ by replacing group orbits in Definition~\ref{def:coerciveG} with level sets $\varphi^{-1}(\varphi(x))$, $x\in M$.
\begin{definition}
\label{def:coercive-phi}
 One shall say that a  continuous map $\varphi:M\to N$  is level-coercive if all its level sets are compact and for every $t>0$ the set 
\begin{equation}
M_t\eqdef\{x\in M: \mathrm{diam}_M\varphi^{-1}(\varphi(x))\le t\}
\end{equation}
is bounded in $M$.
\end{definition}
We use the term {\em level-coercive} because proper semibounded real-valued maps are often called in literature coercive. Level-coercivity is, on the other hand, a property not of a map, but of the equivalence classes of maps with same level sets. For example, an $\R\to(0,1]$-function $x\mapsto e^{-x^2}$ is level-coercive.
\begin{remark}
It may look plausible on the first glance that, like in the case of Theorem~\ref{thm:groupsymm}, level-coercivity of $\varphi$ would yield compactness of embeddings $H_{\varphi}^{1,p}(M)\hookrightarrow L^q(M)$, but this expectation ignores the fact that level bands of a smooth map may exhibit local "bulges" unseen in the case of orbit tubes. In particular, the compactness condition in Theorem~\ref{thm:levels1}, vanishing of $\delta_r$ at infinity, does not follow from level-coercivity, that is, from the condition that diameter of level sets tends to infinity at infinity, as, even in presence of level-coercivity, the ratio in \eqref{eq:deltaalpha} can concentrate at some $y=y(\alpha)$. Consider, for example, $M=\R^2$ with polar coordinates and $\varphi(r,\theta)=r(1+g(r^2\theta))$ for $r>2$, where $g$ is a  smooth function on $\R$ with  $\supp g= [-\pi,\pi]$. 
In order to be able to associate compactness, like in the case of symmetric functions,  with level-coercivity, one needs that level sets of $\varphi$ will have more resemblance to orbits of a compact group. To this end we require them to be compact and their level bands to remain comparably thick (in certain way) at different points.
\end{remark}
\begin{definition}
 We say that the map $\varphi$ has uniformly thick levels if  
there exist $\varepsilon>0$, $r>0$, and an open bounded set $A\subset M$ such that for almost every $z\in \varphi(M)$
  \begin{align}
  \label{eq:unithick}
 \inf_{y\in\varphi^{-1}(z)\setminus A} \int_{\varphi^{-1}(z)\cap B(y,r)}& 
 \frac{\mathrm d\mathcal H_{m-n}(x)}{\mathcal J_\varphi(x)} \ge \\
  & \varepsilon  \sup_{y\in\varphi^{-1}(z)\setminus A}  \int_{\varphi^{-1}(z)\cap B(y,r)}\frac{\mathrm d\mathcal H_{m-n}(x) }{\mathcal J_\varphi(x)} . \nonumber
  \end{align}  
\end{definition}
We draw the following consequence of condition \eqref{eq:unithick}.
\begin{lemma}
\label{lem:farball}
Assume that $\varphi: M\rightarrow N$ is a level-coercive  Lipschitz map 
satisfying \eqref{eq:Jphinz} and \eqref{eq:mualpha}. Moreover, assume that it
 has connected and  uniformly thick levels (i.e. satisfies \eqref{eq:unithick}).  Let $x_0\in M$ and $r>0$. Then for  any $q$, $1\le q <\infty$, we have  
\begin{equation}
\label{eq:farball}
\sigma_R\eqdef\sup_{x\in M\setminus B(x_0,R),\; h\in L^q_\mathrm{loc}(N), h\ge 0, h\neq 0}
\frac
{\int_{B(x,r)}(h\circ\varphi)^q\;d\mathrm{vol}_M}
{\int_{M}(h\circ\varphi)^q\;d\mathrm{vol}_M}\longrightarrow 0
\end{equation}
as $R\to\infty$. 
\end{lemma}
\begin{proof} 
Let $L$ be a Lipschitz constant of $\varphi$.
One can easily see that if $z\notin B_N(\varphi(x), L r)$ then $\varphi^{-1}(z) \cap B(x,r)=\emptyset$. So the coarea formula and \eqref{eq:mualpha}  imply that the integral  $\int_{B(x,r)}(h\circ\varphi)^q\;d\mathrm{vol}_M$ is finite if  $h\in L^q_\mathrm{loc}(N)$. 
If $h\circ \varphi\notin L^{q}(M)$ then the quotient  defining   $\sigma_R$ becomes $0$. So let us  fix $h$ such that $h\circ \varphi\in L^{q}(M)$ and  $h\ge 0$. Let $y_R\notin B(x_0,R)$, $R>0$, and let $z_R\eqdef \varphi(y_R)$.  
Note that $\mathrm{diam} \;\varphi^{-1}(z_R)\to\infty$ as $R\to\infty$, since if the diameters of level sets
$\varphi^{-1}(z_{R_k})$ were bounded on some sequence $R_k\to\infty$, then $d(y_{R_k},x_0)$ would be bounded by level-coercivity of $\varphi$, which contradicts $R_k\to\infty$. 
Thus, since the set $\varphi^{-1}(z_R)$ is connected, there exist points $y^1_R,\dots, y^{j_R}_R\in \varphi^{-1}(z_R)$, $j_R\in\N$, $j_R\to \infty$  as $R\to\infty$, such that the balls $B(y_R^{j},r)$, $j=1,\dots,j_R$, are pairwise disjoint. Indeed, observe first that there exists at least one ball of radius $r$ with a center on $\varphi^{-1}(z_R)$, namely  $B(y_R,r)$. Let $j_R$ be the maximal possible number of pairwise disjoint balls of radius $r$ with centers on $\varphi^{-1}(z_R)$.

Note that for every $R$ sufficiently large 
\begin{equation}
\label{eq:chain_balls}
\min_{j=1,\dots,j_{R}, j\neq i}d(y_{R}^i,y_{R}^j)\le 4r \mbox{ for every }i=1,\dots,j_{R}.
\end{equation}
Indeed, if it were false, then there would exist a $\delta>0$ and an $i$ such that 
$d(y_{R}^{i},y_{R}^j)\ge 4r+2\delta$ whenever $j\neq i$, so by connectedness of the level sets of $\varphi$, 
the boundary $\partial B(y_{R}^{i},2r+\delta)$ will intersect $\varphi^{-1}(z_{R})$ at some $y$, and $B(y,r$) will be disjoint from all $B(y_{R}^j,r)$), contradicting the assumption that $j_R$ is the maximal possible number of disjoint balls centered on $\varphi^{-1}(z_{R})$.

Then if, for some sequence $R_k\to\infty$, $j^*\eqdef \sup_{k\in\N} j_{R_k}<\infty$, then
 from \eqref{eq:chain_balls} would follow $\mathrm{diam}\;\varphi^{-1}(z_{R_k})\le 4rj^*$, which contradicts level-coercivity. Consequently,  for any  positive $h$ such that $h\circ \varphi\in L^q(M)$, using the definition of $j_R$ and \eqref{eq:unithick}, we have
\begin{align*}
\int_M(h\circ\varphi)^q &\;d\mathrm{vol}_M \ge \sum_{j=1}^{j_R}\int_{B(y^{j}_R,r)}(h\circ\varphi)^q\;d\mathrm{vol}_M\ge
\\
&j_R \; \min_{j=1,\dots,j_R}\int_{B(y^{j}_R,r)}(h\circ\varphi)^q\;d\mathrm{vol}_M
=
\\
& j_R \; \min_{j=1,\dots,j_R}
\int_N h(z)^q
\int_{\varphi^{-1}(z)\cap B(y^{j}_R,r)}
 \frac{\mathrm d\mathcal H_{m-n}(y) }{\mathcal J_\varphi(y)}\;d\mathrm{vol}_N(z)\ge
\\
 & j_R \;\varepsilon  \sup_{x\in M\setminus B(x_0,R)}  \int_N h(z)^q
\int_{\varphi^{-1}(z)\cap B(x,r)}
 \frac{\mathrm d\mathcal H_{m-n}(y)}{\mathcal J_\varphi(y)}\;d\mathrm{vol}_N(z)
\end{align*}
 which by \eqref{eq:coarea} gives \eqref{eq:farball} with $\sigma_R\le \frac{1}{\varepsilon j_R}\to 0$ as $R\to\infty$.
\end{proof}
We now can formulate a sufficient condition of compactness in terms of level-coercivity of $\varphi$. 
\begin{theorem}\label{thm:levels2}
 Let $M$ be a complete, connected, $m$-dimensional Riemannian  manifold satisfying (M1) - (M2), let $N$ be a  $n$-dimensional Riemannian manifold, $n<m$. Let  $\varphi:M\to N$ be a   Lipschitz-continuous map satisfying \eqref{eq:Jphinz} and \eqref{eq:mualpha}.
 Assume that $\varphi$ is uniformly thick (i.e. satisfies \eqref{eq:unithick}) and that all level sets of $\varphi$ are connected.

Then, if $\varphi$ is level-coercive,  the subspace  $H^{1,p}_\varphi(M)$, $p\in(1,m)$  is compactly embedded into $L^q(M)$ for every 
$q\in(p,p^*)$. Conversely, if $H^{1,p}_\varphi(M)$, $p\in(1,m)$  is compactly embedded into $L^q(M)$ for some $q\in(p,p^*)$ and injectivity radius of $N$ is positive, then $\varphi$ is level-coercive.
\end{theorem}

\begin{proof} 
\textit{Sufficiency.} By Lemma~\ref{lem:farball}, for any $a>0$
\begin{equation}
\label{eq:localvanish}
\sup_{\stackrel{x\in M\setminus B(x_0,R),}{u\in H_\varphi^{1,2}(M), \|u\|_q\le a}}\int_{B(y,r)}|u|^{q}d\mathrm{vol}_M
\le a^q\sigma_R\to 0\mbox{ as } R\to \infty.
\end{equation}	
Applying this relation to a sequence $u_k\rightharpoonup 0$ with $a=\sup_{k\in\N}\|u_k\|_q$, one may complete the argument exactly as in the proof of Theorem~\ref{thm:levels1}.

\textit{Necessity.} Assume now that $\varphi$ is not level-coercive. 

Let $\delta\in (0,i(N))$, $z_k\in N$ and let $f_k(z)\eqdef (\delta-d_N(z,z_k))_+$, $u_k\eqdef f_k\circ\varphi$. Since the sequence $(f_k)$ is uniformly Lipschitz on $N$, and $\varphi$ is Lipschitz, the sequence $(u_k)$ is uniformly Lipschitz on $M$. 
Since $\varphi$ is not level-coercive, there exists a sequence $z_k\in N$ and $x_k\in\varphi^{-1}(z_k)$
such that $(x_k)$ is discrete and $\mathrm{diam} \;\varphi^{-1}(z_k)$ is bounded. With such choice of $z_k$, taking into account that $|\nabla_Nd_N(z,z_k)|=1$ and $\varphi$ is Lipschitz, the sequence $(u_k)$  is bounded in $H^{1,p}(M)$. Furthermore, its weak  limit point in $H^{1,p}(M)$  vanishes  since its support is of bounded diameter and contains a discrete sequence $(x_k)$.  In order to prove necessity in the theorem it suffices now to show that  none subsequence of $u_k$ does not converge to zero in $L^q(M)$. This would follow once we show that 
$\mathrm{vol}_M(\varphi^{-1}(B_N(z_k,\delta/2))$ is bounded away from zero. Indeed, 
since $\varphi$ is Lipschitz, the set $\varphi^{-1}(B_N(z_k,\delta/2))$ contains a ball $B(x_k,\rho)$ with some $\rho>0$ independent of $k$,
whose measure is bounded away from zero as a consequence of (M1)-(M2). 
\end{proof}
%
%
\begin{remark}
\label{rem:excision}
Note that the assertions of Theorems~\ref{thm:levels1} and \ref{thm:levels2} remain valid if on some bounded subset of $M$ function $\varphi$ is continuous rather than Lipschitz continuous.  
\end{remark}
 \begin{example} Let $M$ be the $m$-dimensional Euclidean space, $1<m$,  and let $N=\R$,   both equipped   with the Euclidean metric. Let $1\le\ell\le\infty$ and $\varphi(x)\eqdef(\sum_{i=1}^m|x_i|^\ell)^{1/\ell}$ unless $\ell=\infty$ and $\varphi(x)\eqdef\max_{i=1,\dots,m}|x_i|$.  
	Then the space $H_\varphi^{1,p}(\R^m)$ consists of $\ell$-radial functions. Then the embedding $H_\varphi^{1,p}(\R^m)\hookrightarrow L^{q}(\R^m)$, $p\in (1,m)$, $q\in (p,p^*)$,  is  compact.
\end{example}
\begin{example} 
\label{ex:1}
Let  $M$ be a complete Riemannian manifold satisfying (M1) - (M2) with a pole at $x_o\in M$, i.e. the map $\exp_{x_o}:T_{x_o}M\rightarrow M$ is a diffeomorphism and $\varphi$ from $M$ to $\R$ be given by  $\varphi(x)=d_M(x_o,x)$ (concerning singularity at $x_o$ cf. Remark~\ref{rem:excision}). In particular, $M$ may be a Cartan-Hadamard manifold i.e. simple-connected manifold that has everywhere non-positive sectional curvature. Then $H_\varphi^{1,p}(M)$ is compactly embedded into $L^{q}(M)$, $p\in(1,m)$, $q\in(p,p^*)$. 
\end{example}
\begin{example} 
	\label{ex:2}
	\item Let $M$ be the $m$-dimensional Euclidean space 
and let $N=\R^n$, $1<n<m$, both equipped   with the Euclidean metric. Let us 
        represent $M=R^m= \R^{\gamma_1}\times \ldots \times \R^{\gamma_n}$ with $\gamma_i\ge 2$. Let
  $r_i\in[1,\infty]$ and let $|\xi|_{r_i}\eqdef \left(\sum_{j=1}^{\gamma_i}|\xi_j|^{r_i}\right)^\frac{1}{r_i}$, $\xi\in\R^{\gamma_i}$ if $r_i<\infty$ and  $|\xi|_{r_i}\eqdef\max_{j=1,\dots,\gamma_i}|\xi_j|$, $\xi\in\R^{\gamma_i}$, if $r_i=\infty$. Let 
       $\varphi(x_1,\dots,x_\nu)=(|x_1|_{r_1},\dots,|x_\nu|_{r_n})$, $x_i\in\R^{\gamma_i}$, $i=1,\dots,n$.   
   Then $H_\varphi^{1,p}(\R^m)$ is compactly embedded into $L^{q}(\R^m)$, $p\in(1,m)$, $q\in(p,p^*)$. 
\end{example}
\begin{example} Let $M$ be the $m$-dimensional Euclidean space 
and let $N=\R^n$, $1<n<m$, both equipped   with the Euclidean metric. Let us 
        represent $M=R^m= \R^{\gamma_1}\times \ldots \times \R^{\gamma_{n-1}}\times \R$ with $\gamma_i\ge 2$, $i=1,\dots,n-1$ . Let
  $r_i\in[1,\infty]$ and let $|\xi|_{r_i}$, $i=1,\dots,n-1$, be as in Example~\ref{ex:2}. Let 
       $\varphi(x_1,\dots,x_n)=(|x_1|_{r_1},\dots,|x_{n-1}|_{r_{n-1}},|x_n|)$, $x_i\in\R^{\gamma_i}$, $i=1,\dots,n-1$.   
   Then the embedding $H_\varphi^{1,p}(\R^m)\hookrightarrow L^{q}(\R^m)$, $p\in(1,m)$, $q\in(p,p^*)$,  is not compact.\end{example}

\section{Extension of compact sets by order}

One can extend the compact subset in $L^q$-spaces  by order not loosing the compactness. 
This was observe in \cite{SS} where the author study subsets consisted of subradial functions belonging to Besov spaces defined on $\R^n$. Here we formulate more general approach for first order Sobolev spaces defined on manifolds. 
\begin{definition} 
	Let $X$ be a $\sigma$-finite metric measure space and let $E$ be a Banach space continuously embedded into $L^q(X)$ for some $q\in(1,\infty)$.  Let   $K\subset E$ be a bounded set in $E$ that is  relatively compact in $L^q(X)$. We say that the  a bounded set  $\tilde K\subset E$ is dominated by $K$ at infinity if there exist a ball $B(x,R)$ in $X$ and a constant $b>0$ such that for any function $u\in \tilde K$ there exists a function $f\in K$ such that $|u(x)|\le b f(x)$ a.e. in $X\setminus B(x_o,R)$ . 	
\end{definition}

\begin{theorem} \label{thm:extord}
	Let $M$ be a complete,  $m$-dimensional, connected, non-compact   Riemannian manifold satisfying (M1) and (M2). Let  $K\subset H^{1,p}(M)$ be  relatively compact in $L^q(M)$,  $p\in[1,m)$ and  $q\in (p,p^*)$. If   $\tilde K$ is  a bounded set in $ H^{1,p}(M)$ dominated by $K$ at infinity, then  $\tilde K$ is also relatively compact in $L^q(M)$. 
\end{theorem}
\begin{proof}
	Let $\{B(x_i,r)\}_{i\in \N}$, $r<r(M)$ be a uniformly locally finite covering of $M$.  Let $\{\varphi_i)\}_{i\in \N}$ be a resolution of unity subordinated to the covering.  Let $(u_k)$ be a sequence in $\tilde K$ and let $I=\{i: B(x_i,r)\cap B(x,R) \not= \emptyset\}$.  The set $I$ is finite since the ball $B(x,R)$ is relatively compact in $M$ and the covering is uniformly finite. 
	Any sequence $(\varphi_i u_k)_k$ has a convergent subsequence in $L^q(M)$, so we can choose a renamed subsequence $u_k$  such that the sequence 
	$\sum_{i\in I} \varphi_i u_k$ is convergent in $L^q(M)$. 
	
	On the other hand by definition of the set $\tilde K$  there exists a sequence $(f_k)$ in $K$ such that $|u_k|\le bf_k$ a.e. on $M\setminus B(x,R)$ for every $k$. 
	So 
	\[ | u_k(x) | = \big| \sum_{i\in \N\setminus I} \varphi_i(x) u_k(x) \big| \le f_k(x) \qquad \text{for a.e.} x\in  M. \]
	Consider a renamed convergent subsequence of $(f_k)$ in $L^q(M)$ and let $f$ be its strong  limit.  Then there exists a function $h\in L^q(M)$ and a renamed subsequence such that $|f_k|\le h$ (see \cite[Theorem~4.9]{Brezis}).
	Then $| u_k|\le bh$. Furthermore, by compactness of local Sobolev embeddings and $\sigma$-finiteness of $M$, a renamed subsequence of $( u_k)$ converges almost everywhere in $M$. Thus, by Lebesgue dominated convergence theorem $( u_k)$ converges in $L^q(M)$.
	Combining the local part with the part at infinity we get that $\tilde K$ is relatively compact. 
\end{proof}

\begin{lemma}
	\label{lem:av}
	 Let $G$ be a compact group of isometries on a Riemannian manifold $M$.
Let $dg$ be the Haar measure on $G$ normalized to $1$.
Then for every $p\in(1,\infty)$ expression $T_Gf(x)\eqdef \int_Gf(\eta x)dG$. 
defines a bounded projection operator of $H^{1,p}(M)$ onto $H_G^{1,p}(M)$  of the norm $1$.	
\end{lemma}
\begin{proof}
	Indeed, by Jensen's inequality and the isometric action of $G$ we have
\begin{align*}
	\int_M &(|\nabla T_Gf (x)|^p+|T_Gf(x)|^p)d\mathrm{vol}\le \\
	&\qquad\int_M \int_G (|\nabla{g x}f( g x)|^p+|f(g x)|^p)\,dg)\,d\mathrm{vol}=
	\\
	&\qquad\int_G \int_M (|\nabla f( g x)|^p+|f(g x)|^p)d\mathrm{vol}\,dg=\|f\|_{H^{1,p}(M)}^p.
\end{align*}
The value $1$ of the norm is attained on every function from the subspace $H_G^{1,p}(M)$. 	
\end{proof}
The following corollary provides compactness of a set of quasisymmetric functions in the sense similar to to \eqref{eq:qsym}.
\begin{corollary}
\label{cor:extorder}
 Let $M$ be a connected, complete, non-compact Riemannian manifold satisfying (M1) and (M2), and let $G$ be a compact, connected and coercive group of isometries on $M$. 
Let $\lambda>1$ and let $B\subset M$ be a compact set in $M$.
	Let $K$ be a bounded subset of $H^{1,p}(M)$, $p\in(1,N)$, consisting of functions satisfying
	\begin{equation}\label{eq:qua-s}
	|f(g x)|\le \lambda|f(x)| 	
	\end{equation} 	
for any $x\in M\setminus B$ and $g\in G$. Then the set $K$ is compact in $L^q(M)$ for every $q\in(p,p^*)$.
  \end{corollary}	
\begin{proof} Let $\chi\in C_0^\infty(M)$ be a $G$-invariant function, $0\le\chi\le 1$, that equals $1$ on $G(B)$, and let 
	\[
	K'=\{(1-\chi)|f|\}_{f\in K}.
	\]
Obviously, $K'$ is bounded in $H^{1,p}(M)$, while 
the $T_G(K')$ is a subset of $H_G^{1,p}(M)$ and is bounded there by Lemma~\ref{lem:av}. Then, by Theorem~\ref{thm:groupsymm}  it is relatively compact in  $L^q(M)$.
By  \eqref{eq:qua-s}, $|f(x)|\le \lambda |f(\eta x)|$ for all $\eta\in G$ and  $x\notin G(B)$. Then $|(1-\chi)f|\le \lambda T_G[(1-\chi)|f|]$ at every $x\in M$. Since $T_G(K')$ is compact, by Theorem~\ref{thm:extord} we have that $K'$ is compact. Applying Theorem~\ref{thm:extord} once again we conclude that $K$  compact. 
\end{proof}

\section{Some variational problems}
In this section we give two elementary existence results for critical points in variational problems.

Let $\Delta_M^p$ denote the Laplace-Beltrami $p$-Laplacian on $M$, given as the Gateux derivative of $\int_M|\nabla u|^pd\mathrm{vol}$. 
\begin{theorem}
\label{thm:varG}
 Let $M$ be a connected, complete, non-compact manifold satisfying (M1) and (M2), let $p\in(1,N)$, $q\in(p,p^*)$, and let $G$ be a compact, connected and coercive group of isometries on $M$. 
	There exists a weak solution $u\in H_G^{1,p}(M)$ to the equation
	\begin{align}\label{eq:eq2}
		-\Delta_M^pu+
		|u|^{p-2}u
		=|u|^{q-2}u,\,\text{ on } M,
	\end{align}
	which is a scalar multiple of a minimum point for 
	\begin{equation}\label{eq:inf2}
	\kappa\eqdef	\inf_{f\in H^{1,p}_G(M):\int_{M}|u|^qd\mathrm{vol}=1}\int_{M} \left(\left|\nabla u\right|^p+|u|^p\right)d \mathrm{vol}.
	\end{equation}
\end{theorem}
\begin{proof}
	Since $H_G^{1,p}(M)$ is embedded into $L^q(M)$, the infimum is positive.  Let $(u_k)$ in $H_G^{1,p}(M)$ be a minimizing sequence, it has a renamed subsequence weakly convergent to some $u_0\in H_G^{1,p}(M)$. By weak semicontinuity of the norm,  $\int_M \left(\left|
	\nabla u_0
	\right|^{p}+|u_0|^p\right)d\mathrm{vol}\le\kappa'$. By Theorem~\ref{thm:groupsymm}  the embedding 
	$H_G^{1,p}(M)\hookrightarrow L^q(M)$ is compact, which implies that $\int_{M}|u_0|^qd\mathrm{vol}=1$. However, by definition of $\kappa$, $\int_M \left(\left|
	\nabla u_0
	\right|^{p}+|u_0|^p\right)d\mathrm{vol}$ cannot be less than $\kappa$ and thus $u_0$ is a minimizer in \eqref{eq:inf1}.
	
The Euler-Lagrange equation for a point of minimum in \eqref{eq:inf2} has the left and a right hand side of \eqref{eq:eq2} equated up to a scalar multiple. Since the left and the right hand sides have different homogeneity degrees, substituting $u$ with $\lambda u$ one can choose $\lambda>0$ to make the multiple equal $1$. 
\end{proof}
\begin{theorem} \label{ex:1}
Let  $M$ be a connected, complete, non-compact manifold satisfying (M1) and (M2) with a pole at $x_o\in M$, i.e. the map $\exp_{x_o}:T_{x_o}M\rightarrow M$ is a diffeomorphism and $\varphi$ from $M$ to $\R$ be given by  $\varphi(x)=d_M(x_o,x)$.  
Let $\Psi(r)=\mathrm{vol}(\partial B(x_o,r))$ (cf \eqref{eq:mualpha}). 
There exists a weak solution $f\in H^{1,p}([0,\infty),\Psi(r)dr)$ of the equation
\begin{align}\label{eq:eq}
	-\frac{1}{\Psi(r)}\frac{d}{dr}\left(\Psi(r)\left|\frac{df}{dr}\right|^{p-2}\frac{df}{dr}\right)=|f|^{q-2}f,\; r>0,
\end{align}
which is a scalar multiple of a minimum point in 
\begin{equation}\label{eq:inf}
\kappa_\varphi\eqdef	\inf_{f,f'\in L^p((0,\infty),\Psi(r) dr):\int_X|f(r)|^q\Psi(r) dr=1}\int_0^\infty 
\left|\frac{df}{dr}\right|^{p}\Psi(r)dr.
\end{equation}
\end{theorem}
\begin{proof}
	Relation \eqref{eq:inf} can be equivalently written as 
	\begin{equation}\label{eq:inf1}
\kappa_\varphi=	\inf_{u\in H_\varphi^{1,p}(M):\int_{M}|u|^qd\mathrm{vol}=1}\int_M \left(\left|
	\nabla u
	\right|^{p}+|u|^p\right)d\mathrm{vol}.
	\end{equation}
The argument for existence of a minimum is now identical to the argument in Theorem~\ref{thm:varG}, once we observe that $H_\varphi^{1,p}(M)$ is compactly embedded into $L^q(M)$ according to Example~\ref{ex:1}.
\end{proof}


\end{document}